\documentclass{article}
\usepackage[T2A]{fontenc}
\usepackage[tbtags]{amsmath}
\usepackage{amsfonts,amssymb,mathrsfs,amscd}
\usepackage{graphicx}

\def\thtext#1{
  \catcode`@=11
  \gdef\@thmcountersep{. #1}
  \catcode`@=12
}

\def\threst{
  \catcode`@=11
  \gdef\@thmcountersep{.}
  \catcode`@=12
}

\newtheorem{theorem}{Theorem}
\newtheorem{lemma}{Lemma}

\newenvironment{remark}{\trivlist \item[\hskip \labelsep{\bf Remark.}]}%
{\endtrivlist}
\newenvironment{proof}{\trivlist \item[\hskip \labelsep{\bf
Proof.}]}{\endtrivlist}

\def\R{{\Bbb R}}
\def\C{{\Bbb C}}
\def\v{\varphi}
\def\G{\Gamma}
\def\c{\circ}
\def\r{\rho}
\def\g{\gamma}
\def\tw{\operatorname{tw}}
\def\sm{\setminus}
\def\cP{{\cal P}}
\def\ss{\subset}

\def\d{\partial}

\def\s{\sigma}
\def\hG{\hat{\Gamma}}
\def\a{\alpha}

\catcode`@=11

\def\.{.\spacefactor\@m}

\catcode`@=12

\begin{document}

\title{The Length of a Minimal Tree With a Given Topology: generalization of Maxwell Formula}
\author{A.\,O.~Ivanov, A.\,A.~Tuzhilin}

\maketitle

  \begin{abstract}
The classic Maxwell formula calculates the length of a planar
locally minimal binary tree in terms of coordinates of its boundary
vertices and directions of incoming edges. However, if an extreme
tree with a given topology and a boundary  has degenerate edges, then the
classic Maxwell formula cannot be applied directly, to
calculate the length of the extreme tree in this case it is necessary to
know which edges are degenerate. In this paper we generalize the
Maxwell formula to arbitrary extreme trees in a Euclidean space of arbitrary
dimension. Now to calculate the length of such a tree, there is no need to
know either what edges are degenerate, or the directions of
nondegenerate boundary edges. The answer is the maximum of some
special linear function on the corresponding compact convex subset of
the Euclidean space coinciding with the intersection of some cylinders.
  \end{abstract}


\section*{Introduction}
 \markright {Introduction}

The present paper is devoted to investigation of extreme trees in
Euclidean spaces. These trees attract interest since the class of such
trees is a natural extension of the set of locally minimal trees and
shortest trees (Steiner minimal trees). The latter ones can be considered
as solutions to Transportation problem and so have grate importance for
applications. It is known, that the searching of a shortest tree spanning
a given boundary set is a very time-consuming algorithmic problem (an
$N\!P$-hard problem), that gives reason for active investigation of
heuristic solutions. As such solutions, one can choose as extreme trees, so
as minimal spanning trees. The latter ones are used often, because there
is a polynomial algorithm of their construction, and quick realizations of
it are well-known and widespread. The relative error of this heuristic in
the worst possible situation is called the Steiner ratio of the ambient
metric space, see~\cite{1}.

At present, the Steiner ratio is not known for any Euclidean spaces,
starting with the two-dimensional plane. Notice that in 60th of the
previous century, E.~N.~Gilbert and H.~O.~Pollak~\cite{1} conjectured that the
Steiner ratio of the Euclidean plane is attained at the vertex set of
a regular triangle and is equal to $\sqrt3/2$. But in spite of many
attempts of different authors (see a review in~\cite{2}), this conjecture is
not proved yet. The most known attempt was taken in 90th by D.~Z.~Du and
F.~K.~Hwang~\cite{3}. But it turns out that their proof contains serious gaps
which were pointed out as by the authors, so as by other specialists. As a
result, currently the conjecture is considered as open.

For the Euclidean spaces of dimension three and more we know even less.
It is proved in~\cite{4} that in these spaces the Steiner ratio is not achieved
at the vertex set of a regular simplex. Moreover, a fast growing low
estimate for a possible number of points in boundary set, the Steiner
ratio could be achieved at is found. Therefore, there is no any reasonable
conjecture concerning the Steiner ratio value for these spaces.

In the present paper a new formula is obtained, which gives an opportunity
to calculate the length of an extreme network spanning a given boundary
set without the network construction. It turns out that the length of such
network can be found as a maximal value of some linear function $\r$ on an
appropriate convex compact subset $S$ of the configuration space $\R^N$.
The function $\r$ depends on the coordinates of the boundary points only,
and the subset $S$ is completely defined in terms of the structure of the
parameterizing tree of the extreme network. It seems to us, that the
formula obtained gives a new view onto relations between the lengths of
distinct extreme trees, in particular, the lengths of shortest trees and
minimal spanning trees.

The authors like to use the opportunity to express their gratitude to
academician A.~T.~Fomenko for his permanent attention to their work.

The work is partly supported by RFBR (project N~10--01--00748), President Program ``Leading Scientific Schools of RF''
(project NSh--3224.2010.1), Euler program of DAAD, and also by Federal
Programs RNP~2.1.1.3704, FCP~02.740.11.5213 and FCP~14.740.11.0794.

\section{Preliminaries}
 \markright {Preliminaries}

Consider an arbitrary tree $G=(V,E)$ with the vertex set $V=\{v_k\}$
containing a fixed subset $B=\{v_1,\ldots,v_n\}$, and with the edge set
$E$. We assume that all the vertices of the tree $G$ having degree $1$ or
$2$ lie in $B$. Such a set $B$ is called a {\em boundary of the tree $G$}
and is denoted by $\d G$. Vertices from $\d G$ and edges incident to such
vertices are said to be {\em boundary}, and the remaining vertices
$v_{n+1},\ldots,v_{n+s}$ and the remaining edges are said to be {\em
interior}.  A {\em network of type $G$} in the Euclidean space $\R^m$ is
an arbitrary mapping $\G\:V\to\R^m$. The restriction of $\G$ onto $\d G$
is called the {\em boundary of the network $\G$} and is denoted by $\d\G$.

Each network $\G$ is useful to represent as a {\em linear graph},
associating each edge $v_kv_l$ of the tree $G$ with the segment
$\bigl[\G(v_k),\G(v_l)\bigr]$ (which could be degenerate). This segment is
called by {\em edge of the network $\G$}. Thereby, an {\em angle\/}
between adjacent non-degenerate edges of the network is naturally defined.
The {\em length of the network\/} is also naturally defined as the sum of
the lengths of all its edges. A network, all whose edges are
non-degenerate, is called {\em non-degenerate}. Besides, an edge of a
network is called {\em boundary\/} ({\em interior\/}), if such is the
corresponding edge of the tree $G$. An edge of the tree $G$ is said to be
{\em $\G$-degenerate\/} ({\em $\G$-non-degenerate\/}), if such is the
corresponding edge of the network $\G$.

Let $\a=\{G_k\}$ be a family of non-intersecting subtrees of $G$.  Define
the {\em reduced\/} tree $G/\a$ as follows.  Its vertices are all the
$G_k$ together with the vertices from $G$, which are not in $G_k$. Join
the vertices from $G/\a$ by an edge, if and only if they are joined in $G$
by an edge.  The {\em boundary\/} of the tree $G/\a$ consists of all its
vertices containing elements from $\d G$.

For each network $\G$, define the {\em $\G$-reduction\/} of the tree $G$
taking the $G_k$ to be equal to the connected components of the subgraph
of $G$ generated by $\G$-degenerate edges of the tree $G$. Notice that the
network $\G$ generates naturally the non-degenerate network of the type
$G/\a$, which is denoted by $\hG$ and called {\em reduced}.

Let $\v\:B\to\R^m$ be a mapping, which is one-to-one with its image.
Consider all possible networks with the boundary $\v$. Then a network
having the least possible length among all such networks is called a {\em
Steiner minimal tree\/} or a {\em shortest network\/} with the given
boundary.

A network $\G$ is said to be {\em locally minimal}, if $\d\G$ is
one-to-one with its image, and the angles between adjacent edges of the
reduced network $\hG$ are at least $120^\c$.  Notice that the vertex
degrees of the network $\hG$ can be equal to $1$, $2$, or $3$, and all its
vertices of degree $1$ and $2$ are boundary. A locally minimal network
whose set of boundary vertices coincides with the set of its vertices of
degree $1$ is called {\em binary}.

Each shortest network is locally minimal. Locally minimal networks, in
turn, are shortest ``in small'', i.e., any sufficiently small part of such
network is a shortest network with the corresponding boundary. Full
information concerning locally minimal networks can be found in~\cite{2}
and~\cite{5}.

For an arbitrary network $\G$ of type $G$ we put $z_k=\G(v_k)\in\R^m$. Let
us identify the boundary of the network $\G$ with the point
$z=(z_1,\ldots,z_n)\in\R^{mn}$ of the configuration space $\R^{mn}$.

By $\langle \cdot,\cdot\rangle$ we denote the Euclidean scalar product in
$\R^{mn}$.  To start with, let $n=2$, $z_1\ne z_2$, and  $\nu$ be the
direction of the segment $[z_1,z_2]$ oriented to $z_2$ (correspondent,
$-\nu$ is the direction of this segment oriented to $z_1$). In the
configuration space we define the vector $\theta=(-\nu,\nu)$. That is, the
vector $\theta$ is composed from the directions of the segment $[z_1,z_2]$
oriented to its consecutive vertices.  Then
 $$
\langle z,\theta\rangle =\langle z_1,-\nu\rangle +\langle z_2,\nu\rangle =
\Bigl\langle z_2-z_1,\dfrac{z_2-z_1}{\|z_2-z_1\|}\Bigr\rangle =\|z_2-z_1\|,
 $$
i.e. $\langle z,\theta\rangle$ is equal to the length of the segment
$[z_1,z_2]$.

Now, let $\G$ be a non-degenerate locally minimal binary network with a
boundary $\v$.  By $\r(\G)$ we denote the length of the network $\G$.  Let
$\nu_k$ be the direction of the boundary edge of the network $\G$ entering
into the point $z_k$, $k\le n$.  Let us define the {\em vector $\theta$ of
the boundary edges directions of the network $\G$} as follows:
$\theta=(\nu_1,\ldots,\nu_n)$. Summing the above expressions for the
length of a segment and taking into account the additivity of the scalar
product and the fact, that the sum of the direction vectors of the edges
entering into a vertex of degree $3$ is equal to zero, we have:
 $$
\r(\G)=\langle z,\theta\rangle .
 $$
The latter equality is referred as {\em Maxwell formula} (see, for
example, \cite{1}, \cite{2}, \cite{6}). Maxwell formula can be naturally generalized to
locally minimal networks of general form. In this case the corresponding
reduced network can contain boundary vertices of degree $2$ or $3$, where
the edges meet by angles of at least $120^\c$. In Maxwell formula for such
networks, at each boundary vertex we need to take the sum of vectors of all
the entering edges.

Notice that for non-degenerate locally minimal binary networks in the
plane, the vector $\theta$ can be calculated by the direction of a single
edge and by the planar structure of the network. It is convenient to do
this in terms of so-called twisting numbers. Let us recall the
corresponding definition. Let the standard orientation of the plane be
fixed, that gives us a possibility to define the positive (left) and
negative (right) rotations. Let $e_1$ and $e_2$ be two edges of a planar
immersed binary tree $G'$, and $\g$ be the unique path in $G'$ joining
them. Then the {\em twisting number $\tw(e_1,e_2)$ from the edge $e_1$ to
the edge $e_2$} is defined as the difference between the numbers of left
(positive) and right (negative) turns in the interior vertices of the path
$\g$ during the motion from $e_1$ to $e_2$ in $G$. Notice that this
function is skew-symmetric and additive along the paths (see~\cite{2} or~\cite{5}).

Let us identify the plane $\R^2$ with the complex field $\C$:  the points
$z_k$ are considered as complex numbers, and the directions $\nu$ are
considered as unit complex numbers $e^{i \psi}$.  Then the configuration
space can be identified with $\C^n$. Instead of the Euclidean scalar
product, here we consider the Hermitian scalar product which is denoted
in the same way.

Let $e_k$ be the unique edge of the network $\G$ entering into the point
$z_k$, $k\le n$, and $\theta=(e^{i\psi_1},\ldots,e^{i\psi_n})$ be the
directions vector of the edges $e_k$. Then the number $\langle
z,\theta\rangle$ is real, and it is equal to the length $\r(\G)$ of the
network $\G$.  Indeed, it suffices to verify this equality for a network
having a single edge $[z_1,z_2]$, $z_1\ne z_2$.  Let $e^{i\psi}$ be
the direction of the segment $[z_1,z_2]$ oriented from $z_1$ to $z_2$
(in this case $-e^{i\psi}$ is the direction of this segment oriented to
the point $z_1$).  Then $\theta=(-e^{i\psi},e^{i\psi})$, and hence
 $$
  \langle z,\theta\rangle =z_1(-e^{-i\psi})+z_2e^{-i\psi}=
  (z_2-z_1)e^{-i\psi} = (z_2-z_1)(\overline{z_2-z_1})/|z_2-z_1|= |z_2-z_1|.
 $$

Let $t_{pq}=\tw(e_p,e_q)$ be the twisting number from the edge $e_p$ to
the edge $e_q$. Then
 $$
e^{i\psi_q}=-e^{i\psi_p}e^{i\frac\pi3 t_{pq}}.
 $$
Put $t_k=\bigl(e^{i\frac\pi3 t_{k1}},\ldots,e^{i\frac\pi3 t_{kn}}\bigr)$.
Notice, that the $k$th component of the vector $t_k$ is equal to $1$. Then
$\theta=-e^{i\psi_k}t_k$. Thus, starting with the direction of one
boundary edge only, one can find the directions of all the remaining
boundary edges (and all the other edges also) of the tree $\G$ using just
the twisting numbers.

As an application of the above formula, let us calculate the length of the
network $\G$ and the directions of all its edges without an explicit
construction of the tree and using only the information concerning the
boundary mapping $\v$ and the planar structure of the network $\G$ (the
twisting numbers for all the pairs of edges of the corresponding immersed
planar tree $G'$).  Since
 $$
\r(\G)=\langle z,\theta\rangle =-e^{i\psi_k}\langle z,t_k\rangle ,
 $$
then $\r(\G)=\bigl|\langle z,t_k\rangle \bigr|$ and $\psi_k$ is equal to
the argument of the number $-\langle z,\theta\rangle /\langle z,t_k\rangle
$, and since $\langle z,\theta\rangle $ is a positive real, then $\psi_k$
is equal to the argument of the number $-\langle t_k,z\rangle $. If we
know $\psi_k$, then we can find out the remaining $\psi_p$ by means of the
above formulas.

To calculate similarly the length and the edges directions of a locally
minimal planar network of general form, we need to partition the reduced
network $\hG$ into the binary components, cutting it by the vertices of
degree $2$ and boundary vertices of degree $3$, and to proceed the
calculation for each component separately.

Let us return to the networks in $\R^m$. By $[G,\v]$ we denote the set of
all the networks in $\R^m$ having the type $G$ and the boundary $\v$.
Each network $\G$ from $[G,\v]$ is uniquely defined by the images
$z_k=\G(v_k)$, $k>n$, of all the interior vertices of $G$. Consequently
writing down the vectors $z_k$, $k>n$, as the components of the vector
$(z_{n+1},\ldots,z_{n+s})$, we identify the set $[G,\v]$ with the space
$\R^{ms}$.  Notice, that the length of the network is a real-valued
function $\r_{G,\v}$ on $[G,\v]$. It is easy to see, that this function is
convex and tends to infinity as the arguments unlimitedly increase.
Therefore, the set of minima of this function is non-empty and convex.
Each network corresponding to a minimum of this function is said to be an
{\em extreme network of the type $G$ with the boundary $\v$}.

Recall that if among the extreme networks of a type $G$ with a boundary
$\v$ a locally minimal network exists, then the extreme network in $[G,\v]$
is unique, see~\cite{2}. Thus, $[G,\v]$ contains at most one locally minimal
network. Notice, that an extreme network need not always be locally
minimal. For example, it can contain vertices of degree more than~$3$.

By definition, the length of an extreme network can be calculated as the
least value of the function $\r_{G,\v}$ on $\R^{ms}$. But this function is
a sum of square roots of square polynomials on the coordinates of the
interior vertices, that complicates the investigation of the extreme
networks in its terms. In the present paper we show, how the length of an
extreme network  can be calculated by maximization of a linear function.
But this maximization need to be proceeded on a more complicated subset
of the configuration space, namely, on an intersection of some cylinders
and a linear subspace. We generalize the Maxwell formula and make it
uniform for all types of networks obtained from a given type $G$ by
degeneration of some edges of the graph $G$. Besides, in the case of
planar networks, we escape the necessity to go over all possible
non-equivalent immersions $G'$ of a binary tree $G$.

\section{Generalized Maxwell formula}
 \markright {Generalized Maxwell formula}

Let $G=(V,E)$ be an arbitrary tree with a boundary
$B=\{v_1,\ldots,v_n\}\ss V$ consisting of $n$ elements. Using the
structure of the tree $G$, let us form a system $S_G$ consisting of
equations and inequalities on variables $\theta_k^j$ which are considered
as the standard coordinates
$(\theta_1^1,\ldots,\theta_1^m,\ldots,\theta_n^1,\ldots,\theta_n^m)$ in
the space $\R^{mn}$.  We put $\theta_k=(\theta_k^1,\ldots,\theta_k^m)$ and
$\theta=(\theta_1,\ldots,\theta_n)$.

Let $e$ be some edge of the tree $G$. By $G_r=(V_r,E_r)$, $r=1,\;2$, we
denote the connected components of the graph $G\sm e$.  Thus, $G\sm
e=G_1\sqcup G_2$, and put $B_r=B\cap V_r$. By $\cP_G(e)$ we denote the
resulting partition $\{B_1,B_2\}$ of the set $B$. Let us choose one of
$B_r\in\cP_G(e)$, and let $B_r=\{v_{k_1},\ldots,v_{k_p}\}$. By $\s_e$ we
denote the inequality $\Bigl\|\sum_{q=1}^p\theta_{k_q}\Bigr\|^2\le 1$, and
by $\s$ we denote the vector equation $\sum_{k=1}^n\theta_k=0$. The system
$S_G$ is formed from $\s$ and $\s_e$ over al edges $e$ of the tree $G$.

Let $|S_G|\ss\R^{mn}$ be the set of all solutions to the system $S_G$.
Notice that $|S_G|$ does not depend on the choice of the components $B_k$,
since equality~$\s$ is valid. Besides, each inequality~$\s_e$ defines a
convex subset in $\R^{mn}$, bounded by an elliptic cylinder, i.e. it is
the product of an $m$-dimensional elliptic disk and $\R^{m(n-1)}$. The
origin $0$ together with the points of the subspace $\Pi$ defined by
equation $\s$, which are close to $0$,  are solutions to the system $S_G$.
Therefore, $|S_G|$ is a convex body in the subspace $\Pi$.

Let $\v\:B\to\R^m$ be an arbitrary mapping, and $z_k=\v(v_k)$, and
$z=(z_1,\ldots,z_n)$. Put $\r_\v(\theta)=\langle \theta,z\rangle $.

 \begin {theorem}\label{Theorem1.}
Under the above notations, the length of each extreme network from
$[G,\v]$ is equal to the largest value of the linear function $\r_\v$ on
the convex set $|S_G|$.
 \end {theorem}

 \begin{proof}
Let $\theta\in|S_G|$ be an arbitrary point of the set $|S_G|$.  To each
pair $(v,e)$, where $e\in E$ is the edge of the tree $G$ incident to the
vertex $v\in V$,  we assign the vector $\theta(v,e)\in\R^m$ using the
vector $\theta$ as follows.  Let $\cP_G(e)=\{B_1,B_2\}$, and $v\in B_1$
and $B_1=\{v_{k_1},\ldots,v_{k_p}\}$. Then we put
 $$
\theta(v,e)=\sum_{q=1}^p\theta_{k_q}.
 $$

 \begin {lemma}\label{Lemma1.}
For any edge $e=vw$ of the tree $G$ we have $\theta(v,e)=-\theta(w,e)$.
 \end {lemma}

 \begin{proof}
The statement of Lemma follows immediately from equality~$\s$.
 \end{proof}

\begin{lemma}\label{Lemma2.}
For any boundary vertex $v_k$, $1\le k\le n$, we have
$\theta_k=\sum\limits_{e:v_k\in e}\theta(v_k,e)$.
 \end {lemma}

 \begin{proof}
Let $e_q=v_kw_q$, $q=1,\ldots,p$, be all the edges from $G$, which are
incident to $v_k$. We cut the tree $G$ by the vertex $v_k$, and let
$G_q=(V_q,E_q)$ be the component containing $w_q$. Put $B_q=V_q\cap B$,
then
 $$
B=B_1\sqcup\cdots\sqcup B_p\sqcup\{v_k\},
 $$
hence
 $$
0=\sum_{q=1}^p\theta(w_q,e_q)+\theta_k=-\sum_{q=1}^p\theta(v_k,e_q)+\theta_k,
 $$
where the last equality follows from Lemma~\ref{Lemma1.}, which was to be proved.
 \end{proof}

 \begin{lemma}\label{Lemma3.}
Let $e_j$, $j=1,\,\ldots,\,r$, be all the edges of the tree $G$, which are
incident to its interior vertex $v$. Then the equality
 $$
\sum_{j=1}^r\theta(v,e_j)=0
 $$
holds.
 \end {lemma}

 \begin{proof}
Indeed, let $e_j=vw_j$, then, due to Lemma~\ref{Lemma1.}, we have:
$\theta(v,e_j)=-\theta(w_j,e_j)$. Let $\cP_G(e_j)=\{B_1^j,B_2^j\}$, and
$v\in B_1^j$ for each $j$.  Then $B=\sqcup_j B_2^j$, and therefore,
 $$
\sum_j\theta(v,e_j)=-\sum_j\theta(w_j,e_j)=-\sum_{k=1}^n\theta_k=0,
 $$
which was to be proved.
 \end{proof}

\begin{lemma}\label{Lemma4.}
For each vertex $v$ of the tree $G$ and each its edge $e$ incident to
$v$, we have $\|\theta(v,e)\|\le1$.
 \end {lemma}

 \begin{proof}
Taking into account equality~$\s(\theta)$,  the statement of Lemma is
equivalent to the validity of inequality $\s_e(\theta)$.
 \end{proof}

Let $\G\in[G,\v]$ e an extreme network, and
$\{v_{n+1},\,\ldots,\,v_{n+s}\}$ be the set of all interior vertices of
the tree $G$. For $n+1\le k\le n+s$ we also define $z_k$ as follows:
$z_k=\G(v_k)$.

 \begin{lemma}\label{Lemma5.}
Under the above notations, we have
 $$
\langle z,\theta\rangle =
\sum\limits_{ \substack{(v_k,e)\\v_k\in e}}
\bigl\langle z_k,\theta(v_k,e)\bigr\rangle .
 $$
 \end {lemma}

 \begin{proof}
Let us partition the sum in the right hand part of the equality
into two sums: the first one is over all interior vertices
$v_k$, and the second one is over all the boundary vertices. The first
sum vanishes due to Lemma~\ref{Lemma3.}. In the second sum we group the terms
corresponding to the same vertex and apply Lemma~\ref{Lemma2.}. Lemma is proved.

Thus, due to Lemma~\ref{Lemma5.},
 $$
\langle z,\theta\rangle =
\sum\limits_{\substack{ (v_k,e)\\ v_k\in e }}
\bigl\langle z_k,\theta(v_k,e)\bigr\rangle =
\sum_{e=v_kv_l}\langle z_k-z_l,\theta(v_k,e)\rangle ,
 $$
therefore, due to Lemma~\ref{Lemma4.}, we have
 $$
\langle z,\theta\rangle =
\sum_{e=v_kv_l}\bigl\langle z_k-z_l,\theta(v_k,e)\bigr\rangle \le
\sum_{e=v_kv_l}\|z_k-z_l\|=\r(\G).
 $$
Since, as we remember, $\theta$ is an arbitrary point from $|S_G|$, we
conclude that the maximal value of the function $\r_\v(\theta)$ on $|S_G|$
does not exceed $\r(\G)$.

Now, let us show that this maximal value is reached. For each
$\G$-nondegenerate edge $e=vw$, by $\xi(v,e)$ we denote the unit vector
from $\R^m$ having the same direction as the edge $e$ of the network $\G$
entering the point $\G(v)$ has. Notice that $\xi(v,e)=-\xi(w,e)$. Further,
due to Extreme Networks Local Structure Theorem (see \cite{2}, Theorem~4.1, or
\cite{5}, Theorem~4.1), for any $\G$-degenerate edge $e=vw$ the pair $(v,e)$
can be assigned with some vector $\xi(v,e)\in\R^m$,
$\bigl\|\xi(v,e)\bigr\|\le1$, in such a way, that the equality
$\xi(v,e)=-\xi(w,e)$ holds and also the vector-valued function $\xi$
(now $\xi$ is defined on all the pairs $(v,e)$ where vertex $v$
is incident to edge $e$) meets relation
 \begin{equation}\label{eq}
\sum_{e:v\in e}\xi(v,e)=0
 \end{equation}
at any {\bf interior\/} vertex $v$ of the tree $G$.  We put
$\xi_k=\sum_{e:v_k\in e}\xi(v_k,e)$, $1\le k\le n$ and show that the
vector $\xi=(\xi_1,\ldots,\xi_n)\in\R^{mn}$ is a solution to the system
$S_G$ and $\r_\v(\xi)=\langle z,\xi\rangle =\r(\G)$.

Indeed, consider an arbitrary edge $e=vw$ of the tree $G$, and let $G\sm
e=G_1\sqcup G_2$, $G_r=(V_r,E_r)$, and $v\in V_1$ and $B_1=B\cap
V_1=\{v_{k_1},\ldots,v_{k_p}\}$.
 \end{proof}

 \begin {lemma}\label{Lemma6.}
Under the above notations, we have $\xi(v,e)=\sum_{q=1}^p\xi_{k_q}$.
 \end {lemma}

 \begin{proof}
At first, assume that $v\in B$ and $v=v_{k_s}$. Since the equality
$\xi(v',e')=-\xi(w',e')$ holds for any edge $e'=v'w'$, we conclude that
 \begin{multline*}
  0=\sum_{e'=v'w'\in E_1}\bigl(\xi(v',e')+\xi(w',e')\bigr)=\\ =
  \sum_{q\in\{1,\ldots,p\}\sm\{s\}}\
 	\sum_{e'\in E_1:v_{k_q}\in e'}\xi(v_{k_q},e')+
 	\sum_{e'\in E_1:v\in e'}\xi(v,e')+\\ +
   	\sum_{v'\in V_1\sm B_1}
 	\sum_{e'\in E_1:v'\in e'}\xi(v',e') =\\ =
 	\sum_{q\in\{1,\ldots,p\}\sm\{s\}}\xi_{k_q}+
 	\bigl(\xi_{k_s}-\xi(v,e)\bigr)+
 \sum_{v'\in V_1\sm B_1}\
 \sum_{e'\in E_1:v'\in e'}\xi(v',e').
 \end{multline*}
Since the latter sum vanishes in accordance to equality~(\ref{eq}), we obtain the
statement of Lemma for the boundary vertex $v$.

Now, let $v$ be an interior vertex of the tree $G$. Then
 \begin{multline*}
 0=\sum_{e'=v'w'\in E_1}\bigl(\xi(v',e')+\xi(w',e')\bigr)=\\ =
 \sum_{q=1}^p\
 \sum_{e':v_{k_q}\in e'}\xi(v_{k_q},e')+
 \sum_{e'\in E_1:v\in e'}\xi(v,e')+\\ +
 \sum_{v'\in V_1\sm\bigl(B_1\cup\{v\}\bigr)}\
 \sum_{e'\in E_1:v'\in e'}\xi(v',e').
 \end{multline*}

The first sum in the latter formula coincides with $\sum_{q=1}^p\xi_{k_q}$
due to definitions, the second one is equal to $-\xi(v,e)$ in accordance
with equality~(\ref{eq}), and the third one vanishes due to equality~(\ref{eq}) also,
thus, we get the statement of Lemma for the interior vertex $v$. Lemma is
proved.
 \end{proof}

 \begin {lemma}\label{Lemma7.}
The equality $\sum_{k=1}^n\xi_k=0$ holds.
 \end {lemma}

 \begin{proof}
Again, taking into account the fact that the equality $\xi(v,e)=-\xi(w,e)$
is valid for any edge $e=vw$, we get
 $$
0=\sum_{e=vw\in E}\bigl(\xi(v,e)+\xi(w,e)\bigr)=
 \sum_{k=1}^n\ \sum_{e:v_k\in e}\xi(v_k,e)+
 \sum_{v\in (V\sm B)}\ \sum_{e\in E:v\in e}\xi(v,e).
 $$
The first sum in the latter equality coincides with $\sum_{k=1}^n\xi_k$
due to definitions, and the second one vanishes due to equality~(\ref{eq}). Lemma
is proved.
 \end{proof}

Lemma~\ref{Lemma6.} and the condition $\|\xi(v,e)\|\le1$ imply that $\xi$ meets each
inequality $\s_e$. Lemma~\ref{Lemma7.} implies that $\xi$ also meets condition~$\s$.
Thus, $\xi\in|S_G|$. Further,
 \begin{multline*}
\r_\v(\xi)=\langle z,\xi\rangle =\sum_{k=1}^n\langle z_k,\xi_k\rangle
=\sum_{k=1}^n\sum_{e:v_k\in e}\bigl\langle z_k,\xi(x_k,e)\bigr\rangle=\\
=\sum_{k=1}^{n+s}\sum_{e:v_k\in e}\bigl\langle z_k,\xi(x_k,e)\bigr\rangle
	=\sum_{e=v_kv_l\in E}\bigl\langle z_k-z_l,\xi(v_k,e)\bigr\rangle=\\
	=\sum\limits_{\substack{e=v_kv_l\in E\\ z_k\ne z_l}}	\bigl\langle z_k-z_l,\xi(v_k,e)\bigr\rangle
	=\sum\limits_{\substack{e=v_kv_l\in E\\z_k\ne z_l}}\biggl\langle z_k-z_l,\frac{z_k-z_l}{\|z_k-z_l\|}\biggr\rangle=\\
	=\sum\limits_{\substack{e=v_kv_l\in E\\z_k\ne z_l}}\|z_k-z_l\|=\r(\G).
 \end{multline*}
Theorem~\ref{Theorem1.} is proved.
 \end{proof}

 \begin{remark}
Notice that, due to Lemmas~\ref{Lemma2.} and~\ref{Lemma4.}, each component of vector
$\theta\in|S_G|$ is bounded, therefore $|S_G|$ is a compact subset of
$\R^{mn}$. Thus, taking the above into account, we conclude that $|S_G|$
is a convex compact.
 \end{remark}

As above, let $G=(V,E)$ be a tree with a boundary $B=\{v_1,\ldots,v_n\}\ss
V$ and $E_d\ss E$ be some family of edges of the tree $G$. By $\a=\{G_k\}$
we denote the family of nonintersecting subtree of $G$, such that the
union of all their edges coincides with $E_d$. Assume that each $G_k$
intersects $B$ by at most one vertex. Such families $E_d$ are said to be
{\em admissible}. Consider the reduced tree $G/\a$ and choose its boundary
to be the set of all the vertices which intersect the set $B$. Due to the
restrictions imposed on the subtree $G_k$ from $\a$, the boundary set
obtained consists of the same number of points as the set $B$ does, that
gives us an opportunity to identify it with $B$. Thus, for an arbitrary
mapping $\v\: B\to\R^m$, the next two spaces are defined:  $[G,\v]$ and
$[G/\a,\v]$. If $\G\in[G,\v]$ degenerates all the edges from $E_d$, then
$\G$ generates naturally the network from $[G/\a,\v]$, which we denote by
$\G/\a$.

As above, consider the system $S_G$ consisting of equation~$\s$ and
inequalities~$\s_e$, and through out of it all the inequalities
corresponding to the edges $e\in E_d$.  By $S_G\sm E_d$ we denote the
resulting system.

 \begin {theorem}\label{Theorem2.}
Under the above notations, the length of extreme network
$\G/\a\in[G/\a,\v]$ is equal to the largest value of the linear function
$\r_\v$ on the set $|S_G\sm E_d|$ of all the solutions to the system
$S_G\sm E_d$.
 \end {theorem}

 \begin{proof}
It is sufficient to notice that $S_G\sm E_d=S_{G/\a}$, since for any
edge $e$ from $G/\a$ we have: $\cP_{G/\a}(e)=\cP_{G}(e)$. Theorem~\ref{Theorem2.} is
proved.
 \end{proof}

Thus, throwing out of inequalities of the form $\s_e$ from the system
$S_G$ is equivalent to factorization of the tree $G$ by the corresponding
edges $e$. The authors hope that the formula obtained for the calculation
of extreme networks lengths can be useful in investigation of Steiner
ratio of Euclidean spaces.

 \begin{remark}
The results of this paper could be generalized in the following
directions. First, Extreme Networks Local Structure Theorem from~\cite{2}
and~\cite{5} is proved for the case of extreme weighted networks. The classical
Maxwell formula also can be easily transformed to this case (namely, the
length of the direction vector of a weighted edge should be equal to the
weight of the edge).  Therefore, the above results can be easily
generalized to the case of weighted trees.

Second, the classical Maxwell formula remains valid for extreme networks
with cycles. Therefore it seems not very difficult to generalize
Theorems~\ref{Theorem1.} and~\ref{Theorem2.} to the case of arbitrary extreme networks in $\R^m$ (not
necessary the trees as in the present paper), and also to the case of
arbitrary weighted extreme networks in $\R^m$.

At last, in~\cite{2} we obtained theorems on the local structure of extreme
networks in normed spaces.  It would be interesting to work out an
analogue to the Maxwell formula for this case in terms of so-called
$\r$-impulse, see~\cite{2}, and after that to generalize the results of the
present paper to the case of normed spaces.
  \end{remark}


 \end {document}